\documentclass{amsart}
\usepackage{amssymb, amsmath, mathrsfs, verbatim, url, bbm}

\theoremstyle{plain}
\newtheorem{theorem}{Theorem}[section]

\newtheorem{proposition}[theorem]{Proposition}
\newtheorem{corollary}[theorem]{Corollary}

\theoremstyle{definition}

\newtheorem{question}[theorem]{Question}

\newcommand{\dom}{\mathsf{dom}}
\newcommand{\ran}{\mathsf{ran}}

\newcommand{\maxi}{\mathsf{max}}
\newcommand{\id}{\mathsf{id}}
\newcommand{\cf}{\mathsf{cf}}

\newcommand{\Gd}{\mathsf{G_\delta}}
\newcommand{\CDH}{\mathsf{CDH}}

\newcommand{\CH}{\mathsf{CH}}
\newcommand{\ZFC}{\mathsf{ZFC}}
\newcommand{\MAsigma}{\mathsf{MA}(\sigma\text{-}\mathrm{centered})}
\newcommand{\MAc}{\mathsf{MA}(\mathrm{countable})}
\newcommand{\SLH}{\mathsf{SLH}}

\newcommand{\BB}{\mathcal{B}}

\newcommand{\CC}{\mathcal{C}}
\newcommand{\DD}{\mathcal{D}}

\newcommand{\HH}{\mathcal{H}}

\newcommand{\PPP}{\mathbb{P}}
\newcommand{\RRR}{\mathbb{R}}
\newcommand{\QQQ}{\mathbb{Q}}
\newcommand{\cccc}{\mathfrak{c}}

\begin{document}

\title{Products and countable dense homogeneity}

\author{Andrea Medini}

\address{Kurt G\"odel Research Center for Mathematical Logic
\newline\indent University of Vienna, W\"ahringer Stra{\ss}e 25, A-1090 Wien,
Austria}

\email{andrea.medini@univie.ac.at}

\urladdr{http://www.logic.univie.ac.at/\~{}medinia2/}

\date{June 10, 2014}

\thanks{The author acknowledges the support of the FWF grant I 1209-N25.}

\begin{abstract}
Building on work of Baldwin and Beaudoin, assuming Martin's Axiom, we construct
a zero-dimensional separable metrizable space $X$ such that $X$ is countable dense homogeneous while $X^2$ is
not. It follows from results of Hru\v{s}\'ak and Zamora Avil\'es that such a
space $X$ cannot be Borel. Furthermore, $X$ can be made homogeneous and
completely Baire as well.
\end{abstract}

\maketitle

\section{Introduction}

As it is common in the literature about countable dense homogeneity, by
\emph{space} we will always mean ``separable metrizable topological space''. By
\emph{countable} we will always mean ``at most countable''. For all undefined
topological notions, we refer to \cite{vanmilli}. Our reference for descriptive set theory is
\cite{kechris}. For all other set-theoretic notions, we refer to \cite{kunen}.
Given a space $X$, we will denote by $\HH(X)$ the group of homeomorphisms of
$X$. Recall that a space $X$ is \emph{countable dense homogeneous} (briefly,
$\CDH$) if for every pair $(A,B)$ of countable dense subsets of $X$ there exists
$h\in\HH(X)$ such that $h[A]=B$.

The fundamental positive result in the theory of $\CDH$ spaces is the following
(see \cite[Theorem 5.2]{andersoncurtisvanmill}). In particular, it shows that the Cantor
set $2^\omega$, the Baire space $\omega^\omega$, the Euclidean spaces $\RRR^n$,
the spheres $S^n$ and the Hilbert cube $[0,1]^\omega$ are all examples of $\CDH$
spaces. See \cite[\S14, \S15 and \S16]{arkhangelskiivanmill} for much more on this topic. Recall that a space is
\emph{strongly locally homogeneous} (briefly, $\SLH$) if there exists a base
$\BB$ for $X$ such that for every $U\in\BB$ and $x,y\in U$ there exists
$h\in\HH(X)$ such that $h(x)=y$ and $h\upharpoonright (X\setminus
U)=\id_{X\setminus U}$.

\begin{theorem}[Anderson, Curtis, Van Mill]
Every Polish $\SLH$ space is $\CDH$.
\end{theorem}

Using the fact that homeomorphisms permute the connected components, it is easy
to see that the product $2^\omega\times S^1$ is not $\CDH$. Therefore, countable
dense homogeneity is not productive, not even in the class of compact topological
groups. However, the situation improves if we restrict our attention to
zero-dimensional spaces. Let
$$
\CC=\{X:X\approx\kappa\oplus (\lambda\times 2^\omega)\oplus
(\mu\times\omega^\omega)\textrm{, where }0\leq\kappa,\lambda,\mu\leq\omega\}
$$
be the class of spaces that are homeomorphic to a countable disjoint sum of
copies of $2^\omega$, $\omega^\omega$ and $1$. Recall that a space is an
\emph{absolute Borel set} (or simply \emph{Borel}) if it is a Borel subspace of
some Polish space. See \cite[Lemma 2.2, Corollary 2.4 and Corollary 2.5]{hrusakzamoraaviles}
for a proof of the following theorem.
\begin{theorem}[Hru\v{s}\'ak, Zamora Avil\'es]\label{classify}
If $X$ is a zero-dimensional Borel $\CDH$ space, then $X\in\CC$.
\end{theorem}
\begin{proposition}
The class $\CC$ is closed under countable products.
\end{proposition}
\begin{proof}
It is easy to verify directly that $\CC$ is closed under finite products. So let
$X$ be the product of a countably infinite subcollection of $\CC$. Without loss of
generality, we can assume that all factors of $X$ are non-empty and infinitely
many of them have size bigger
than $1$. It follows that $X$ is a zero-dimensional Polish space with
no
isolated points. Therefore, $X$ is homeomorphic to $2^\omega$ (if $X$ is compact)
or to $\omega^\omega$ (if $X$ is not compact, hence nowhere compact).
\end{proof}
\begin{corollary}\label{positive}
Assume that $I$ is countable and $X_i$ is a zero-dimensional Borel
$\CDH$ space for every $i\in I$. Then $\prod_{i\in I}X_i$ is $\CDH$.
\end{corollary}

Using a method of Baldwin and Beaudoin (see \S2), we will show that the
``Borel'' assumption in Corollary \ref{positive} cannot be dropped. The following is our main result.
The construction of the example is the content of \S3, and the
verification of its properties is the content of \S4. Recall that a space
$X$ is \emph{completely Baire} if every closed subspace of $X$ is Baire. By a
classical result of Hurewicz, a space is completely Baire if and only if it does
not contain any closed subspace that is homeomorphic to $\QQQ$ (see \cite[Corollary 1.9.13]{vanmilli}).

\begin{theorem}\label{main}
Assume $\MAsigma$. Then there exists a zero-dimensional $\CDH$ space $X$ such
that $X^2$ is not $\CDH$.\,\footnote{\,It follows from recent results of
Hern\'andez-Guti\'errez, Hru\v{s}\'ak and Van Mill that such a space also exists
in models obtained by adding at least $\omega_2$ Cohen reals to a model of
$\CH$. Just consider $X=Y\oplus 2^\omega$, where $Y$ is a meager $\CDH$ subspace
of
$2^\omega$ (such a space exists in $\ZFC$ by \cite[Theorem 4.1]{hernandezgutierrezhrusakvanmill}).
Notice that $Z=X^2\setminus (2^\omega\times 2^\omega)$ is meager and has size
$\cccc$, so it is not $\CDH$ by the proof of \cite[Theorem 4.4]{hernandezgutierrezhrusakvanmill}.
On
the other hand, $Z$ is preserved by every homeomorphism of $X^2$ because it is
the
union of all meager open subsets of $X^2$. Therefore $X^2$ is not
$\CDH$. However, it is clear that $X$ is neither homogeneous nor Baire.}
Furthermore, $X$ can be made homogeneous and
completely Baire as well.

\end{theorem}
\noindent We will not prove the second part of the theorem, since it can be
obtained by exactly the same methods used in the proof of \cite[Theorem 3.5]{baldwinbeaudoin}, and it would make our construction unnecessarily cumbersome. In
fact, those methods show that the space $X$ can be made a homogeneous Bernstein set. Recall that a subspace $X$ of
$2^\omega$ is a \emph{Bernstein set} if $X\cap K\neq\varnothing$ and $(2^\omega\setminus X)\cap
K\neq\varnothing$ for every perfect subset $K$ of $2^\omega$. Using the
characterization mentioned above, it is easy to see
that every Bernstein set is completely Baire.

We conclude this introduction with several open questions.
\begin{question}
Can the assumption of $\MAsigma$ in Theorem \ref{main} be dropped?
\end{question}

\begin{question}\label{general}
For which $\kappa$ such that $2\leq\kappa\leq\omega$ is there a
zero-dimensional space $X$ such that $X^n$ is $\CDH$ for every $n<\kappa$ while
$X^\kappa$ is not?\,\footnote{\,The case $\kappa=\omega$ was recently settled by
Hern\'andez-Guti\'errez, Hru\v{s}\'ak and Van Mill, who proved the existence of
such a space in $\ZFC$ (see \cite[Theorem 4.8]{hernandezgutierrezhrusakvanmill}).} Can $X$ be homogeneous and completely Baire? 
\end{question}
\noindent Notice that the space $X=2^\omega\oplus S^1$ is $\CDH$ while $X^2$ is
not. However, in the case $3\leq\kappa\leq\omega$, we would not know the answer
to Question \ref{general} even if the zero-dimensionality requirement were
dropped.

The \emph{type} of a countable dense subset $D$
of a space $X$ is $\{h[D]:h\in\HH(X)\}$. Clearly, a space is $\CDH$ if and only
if it has exactly one type of countable dense subsets. Also notice that $\cccc$
is the maximum possible number of types of countable dense subsets of a space.
In \cite{hrusakvanmill},
Hru\v{s}\'ak and Van Mill started an investigation of this natural notion. In
particular, \cite[Theorem 4.5]{hrusakvanmill} gives
a condition under which a space must have $\cccc$ types of countable dense
subsets (see also \cite[Theorem 14 and Theorem 16]{kunenmedinizdomskyy} for more specific
statements). However, we were unable to answer the following question, even in
the case $\kappa=\cccc$.
\begin{question}
For which cardinals $\kappa$ such that $2\leq\kappa\leq\cccc$ is there a
zero-dimensional $\CDH$ space $X$
such that $X^2$ has exactly $\kappa$ types of countable dense subsets?
\end{question}

By \cite[Corollary 2.7]{hrusakzamoraaviles}, Theorem \ref{classify} extends to all
projective spaces if one assumes the axiom of Projective Determinacy. Hence, the same is true
for Corollary \ref{positive}. Therefore, the following question seems natural.
\begin{question}
Is it consistent that there exists a zero-dimensional analytic $\CDH$ space $X$
such that $X^2$ is not $\CDH$? Coanalytic?
\end{question}

By considering $(\omega+1)\times 2^\omega\approx 2^\omega$, one sees that a
factor of a $\CDH$ product need not be $\CDH$. Actually, in \cite{vanmillr}, Van Mill constructed a rigid space $X$ such that $X^2\approx [0,1]^\omega$. Recall
that a space is \emph{rigid} if its only homeomorphism is the identity. In particular, $X$ is a continuum that is not $\CDH$ while $X^2$ is
$\CDH$.
\begin{question}
For which $\kappa$ such that $2\leq\kappa\leq\omega$ is there a space
$X$ such that $X^n$ is not $\CDH$ for every $n<\kappa$ while $X^\kappa$ is
$\CDH$?\,\footnote{\,Notice that $X=\omega+1$ and $X=[0,1]^d$ for every $d$ such that $1\leq d<\omega$ answer in the
affirmative the case $\kappa=\omega$.} Can $X$ be a continuum?
\end{question}

In \cite{lawrence}, Lawrence constructed a non-trivial rigid zero-dimensional
space $X$ such that $X^2$ is homogeneous. But we do not know whether $X^2$ can be
made $\CDH$.
\begin{question}
Is there a zero-dimensional space $X$ that is not $\CDH$ while
$X^2$ is $\CDH$? Can $X$ be rigid?
\end{question}

\section{Results of Baldwin and Beaudoin}
Given an infinite cardinal $\lambda$, a subset $D$ of a space $X$
is \emph{$\lambda$-dense} in $X$ if $|D\cap U|=\lambda$ for every non-empty open
subset $U$ of $X$. The following results are \cite[Lemma 3.1 and Lemma 3.2]{baldwinbeaudoin}. We present a simpler version of the proof of the first result.
Similar posets were recently used in the proofs of \cite[Lemma 22 and Lemma 25]{medinimilovich}.

\begin{theorem}[Baldwin, Beaudoin]\label{bbpreliminary}
Assume $\MAsigma$. Let $\kappa<\cccc$ be a cardinal. Suppose that $A_\alpha$ and
$B_\alpha$ are countable dense subsets of $2^\omega$ for each $\alpha < \kappa$.
Also assume that $A_\alpha\cap A_\beta=\varnothing$ and $B_\alpha\cap
B_\beta=\varnothing$ whenever $\alpha<\beta<\kappa$. Then there exists
$f\in\HH(2^\omega)$ such that $f[A_\alpha]=B_\alpha$ for every $\alpha <
\kappa$.
\end{theorem}
\begin{proof}
Consider the poset $\PPP$ consisting of all pairs of the form
$p=(g,\pi)=(g_p,\pi_p)$ such that, for some $n=n_p\in\omega$, the following
requirements are satisfied. Let $A=\bigcup_{\alpha\in\kappa}A_\alpha$ and
$B=\bigcup_{\alpha\in\kappa}B_\alpha$.
\begin{itemize}
\item $g=g_{\alpha_1}\cup\cdots\cup g_{\alpha_m}$, where
$\alpha_1<\cdots<\alpha_m<\kappa$ and each $g_{\alpha_i}$ is a finite bijection
from $A_{\alpha_i}$ to $B_{\alpha_i}$.
\item $\pi$ is a permutation of ${}^{n}2$.
\item $\pi(a\upharpoonright n)=g(a)\upharpoonright n$ for every $a\in\dom(g)$.
\end{itemize}

\noindent Order $\PPP$ by declaring $q\leq p$ if the following conditions are satisfied.
\begin{itemize}
\item $g_q\supseteq g_p$.
\item $\pi_q(t)\upharpoonright n_p=\pi_p(t\upharpoonright n_p)$ for all $t\in
{}^{n_q}2$.
\end{itemize}

For each $\ell\in\omega$, define
$$
D_\ell=\{p\in\PPP:n_p\geq\ell\}.
$$
Let $p=(g,\pi)\in\PPP$ with $n_p=n$, and let $\ell\in\omega$. Choose $n'\geq\ell,n$
big enough so that all $a\upharpoonright n'$ are distinct for $a\in\dom(g)$ and
all $b\upharpoonright n'$ are distinct for $b\in\ran(g)$. Now it is easy to
obtain a permutation $\pi'$ of ${}^{n'}2$ such that $q=(g,\pi')\in\PPP$ and
$q\leq p$. So each $D_\ell$ is dense in $\PPP$.

For each $a\in A$, define
$$
D^{\dom}_a=\{p\in\PPP:a\in\dom(g_p)\}.
$$
Given $p\in\PPP$ and $a\in A_\alpha\setminus \dom(g_p)$, one can simply choose
$b\in B_\alpha\setminus \ran(g_p)$ such that $b\upharpoonright
n_p=\pi_p(a\upharpoonright n_p)$. This choice will make sure that
$q=(g_p\cup\{(a,b)\},\pi_p)\in\PPP$. Furthermore, it is clear that $q\leq p$. So
each $D^{\dom}_a$ is dense in $\PPP$. 

For each $b\in B$, define
$$
D^{\ran}_b=\{p\in\PPP:b\in\ran(g_p)\}.
$$
As above, one can easily show that each $D^{\ran}_b$ is dense in $\PPP$.

It remains to show that $\PPP$ is $\sigma\text{-}\mathrm{centered}$. We will
proceed as in \cite[Exercise III.2.13]{kunen}. It will be enough to construct
$x_e:A\longrightarrow B$ for $e\in\omega$ such that $g\subseteq x_e$ for some
$e$ whenever $g=g_p$ for some $p\in\PPP$. Let $\{f_\alpha:\alpha<\kappa\}$ be an
independent family of functions (see \cite[Exercise III.2.12]{kunen}). In
particular, each $f_\alpha:\omega\longrightarrow\omega$ and, given any
$j_1,\ldots,j_m\in\omega$ and $\alpha_1<\cdots<\alpha_m<\kappa$, there exists
$e\in\omega$ such that $f_{\alpha_1}(e)=j_1,\ldots, f_{\alpha_m}(e)=j_m$.
Enumerate as $\{d^\alpha_j:j\in\omega\}$ all finite bijections from $A_\alpha$
to $B_\alpha$. It is easy to check that defining
$$
x_e=\bigcup_{\alpha\in\kappa}d^\alpha_{f_\alpha(e)}
$$
for every $e\in\omega$ yields the desired functions. (Notice that $\PPP$ would
be $\sigma\text{-}\mathrm{centered}$ even if $\kappa=\cccc$. However, in that
case, we would have too many dense sets.)

Since $|A|<\cccc$ and $|B|<\cccc$, the collection of dense sets
$$
\DD=\{D_\ell:\ell\in\omega\}\cup\{D^\dom_a:a\in A\}\cup\{D^\ran_b:b\in B\}
$$
has also size less than $\cccc$. Therefore, by $\MAsigma$, there exists a
$\DD$-generic filter $G\subseteq\PPP$. To define $f(x)(i)$, for a given $x\in
2^\omega$ and $i\in\omega$, choose any $p\in G$ such that $i\in n_p$ and set
$f(x)(i)=\pi_p(x\upharpoonright n_p)(i)$.
\end{proof}
\begin{corollary}[Baldwin, Beaudoin]\label{bb}
Assume $\MAsigma$. Let $\kappa<\cccc$ be a cardinal. Suppose that $A_\alpha$ and
$B_\alpha$ are $\lambda_\alpha$-dense subsets of $2^\omega$ for each $\alpha <
\kappa$, where each $\lambda_\alpha<\cccc$ is an infinite cardinal. Also assume
that $A_\alpha\cap A_\beta=\varnothing$ and $B_\alpha\cap B_\beta=\varnothing$
whenever $\alpha<\beta<\kappa$. Then there exists $f\in\HH(2^\omega)$ such that
$f[A_\alpha]=B_\alpha$ for every $\alpha < \kappa$.
\end{corollary}
\begin{proof}
Notice that each $A_\alpha$, being a $\lambda_\alpha$-dense subset of
$2^\omega$, can be partitioned into $\lambda_\alpha$ countable dense subsets of
$2^\omega$. The same holds for each $B_\alpha$. Since $\MAsigma$ implies that
$\cccc$ is regular (see for example \cite[Theorem III.3.61 and Lemma III.1.26]{kunen}), we can apply Theorem \ref{bbpreliminary}.
\end{proof}

The results in this section were originally employed by Baldwin and Beaudoin to
construct a homogeneous $\CDH$ Bernstein set under $\MAsigma$ (see \cite[Theorem 3.5]{baldwinbeaudoin}). We remark that their proof contains a small
inaccuracy. Using their notation, given $\lambda$-dense $A,B\subseteq X_\alpha$,
it is not possible to use Corollary \ref{bb} to find $g\in\HH(2^\omega)$ such
that $g[A]=B$, $g[B]=A$, $g[X_\gamma\setminus(A\cup B)]=X_\gamma\setminus(A\cup
B)$ and $g[Y_\gamma]=Y_\gamma$, since $A$ and $B$ might not be disjoint.
However, this is easily fixed by requiring instead that $g[A]=B$,
$g[X_\gamma\setminus A]=X_\gamma\setminus B$ and $g[Y_\gamma]=Y_\gamma$, as we
do in the next section.

\section{The construction}

Let $Q$ and $R$ be any two disjoint countable dense subsets of $2^\omega$. Given
$i\in 2$, denote by $\pi_i:2^\omega\times 2^\omega\longrightarrow
2^\omega$ the natural projection on the $i$-th coordinate. Given $S\subseteq
2^\omega$ and a subgroup $\HH$ of $\HH(2^\omega)$, let
$$
\HH[S]=\{h(z):z\in S, h\in\HH\}
$$
denote the closure of $S$ under the action of $\HH$.

Enumerate as $\{(A_\alpha, B_\alpha):\alpha<\cccc\}$ all pairs of countable
dense subsets of $2^\omega$, making sure that each pair is listed cofinally
often.
Enumerate as $\{g_\alpha:\alpha<\cccc\}$ all homeomorphisms satisfying the
following conditions.

\begin{itemize}

\item $g_\alpha:T_\alpha\longrightarrow T_\alpha$, where $T_\alpha$ is a
$\Gd$ subset of $2^\omega\times 2^\omega$.

\item $Q^2\subseteq T_\alpha$.

\item $\pi_0\upharpoonright (g_\alpha[Q^2])$ is injective.

\end{itemize}

\noindent Notice that each $T_\alpha$ is dense in $2^\omega\times 2^\omega$.
In particular, if $M$ is meager in $2^\omega\times 2^\omega$ then $M\cap
T_\alpha$ is meager in $T_\alpha$. Also notice that each $T_\alpha$ is a Polish
space.

By transfinite recursion, we will construct two increasing sequences $\langle
X_\alpha:\alpha<\cccc\rangle$ and $\langle
Y_\alpha:\alpha<\cccc\rangle$ of subsets of $2^\omega$, and an increasing
sequence $\langle \HH_\alpha:\alpha<\cccc\rangle$ of subgroups of $\HH(2^\omega)$.

By induction, we will make sure that the following requirements are satisfied
for every $\alpha<\cccc$.

\begin{enumerate}

\item\label{small} $|X_\alpha|,|Y_\alpha|,|\HH_\alpha|\leq
\maxi\{|\alpha|,\omega\}$.

\item\label{disjoint} $X_\alpha\cap Y_\alpha=\varnothing$.

\item\label{preservation} $h[X_\alpha]=X_\alpha$ and $h[Y_\alpha]=Y_\alpha$ for
all $h\in\HH_\alpha$.

\item\label{density} $X_{\alpha+1}\setminus X_\alpha$ and $Y_{\alpha+1}\setminus
Y_\alpha$ are
$\maxi\{|\alpha|,\omega\}$-dense in $2^\omega$.

\item\label{cdh} If $A_\alpha\cup B_\alpha\subseteq X_\alpha$ and
$X_\alpha\setminus (A_\alpha\cup B_\alpha)$ is $\maxi\{|\alpha|,\omega\}$-dense
in $2^\omega$ then there exists $f\in\HH_{\alpha+1}$ such that
$f[A_\alpha]=B_\alpha$.

\item\label{kill} There exists $(x,y)\in X_{\alpha+1}^2\cap T_\alpha$ such that
$\pi_0(g_\alpha(x,y))\in Y_{\alpha+1}$.

\end{enumerate}

Start the construction by letting $X_0=Q$, $Y_0=R$ and
$\HH_0=\{\id_{2^\omega}\}$.

Take unions at limit stages. At a successor stage $\alpha+1$, assume that
$X_\beta$, $Y_\beta$ and $\HH_\beta$ are given for every $\beta\leq\alpha$. We
will start by defining $\HH_{\alpha+1}$,
making sure that condition $(\ref{cdh})$ is satisfied. Let
$\lambda=\maxi\{|\alpha|,\omega\}$. If $(A_\alpha\cup
B_\alpha)\nsubseteq X_\alpha$ or $X_\alpha\setminus (A_\alpha\cup B_\alpha)$ is
not $\lambda$-dense in $2^\omega$, simply let $\HH_{\alpha+1}=\HH_\alpha$. Now
assume
that $(A_\alpha\cup B_\alpha)\subseteq X_\alpha$ and $X_\alpha\setminus
(A_\alpha\cup B_\alpha)$ is $\lambda$-dense in $2^\omega$. By applying Corollary
\ref{bb} with $\kappa=3$, $\lambda_0=\omega$ and
$\lambda_1=\lambda_2=\lambda$, one obtains $f\in\HH(2^\omega)$ such that
$f[A_\alpha]=B_\alpha$, $f[X_\alpha\setminus A_\alpha]=X_\alpha\setminus
B_\alpha$ and $f[Y_\alpha]=Y_\alpha$. Let
$\HH_{\alpha+1}=\langle\HH_\alpha\cup\{f\}\rangle$. For the rest of the
proof, let $\HH=\HH_{\alpha+1}$.

Next, we will make sure that condition $(\ref{kill})$ is satisfied. Define
$$
C_h^i=\{(x,y)\in T_\alpha:h(\pi_i(x,y))=\pi_0(g_\alpha(x,y))\}
$$
for $h\in\HH$ and $i\in 2$. Clearly, each $C_h^i$ is closed.
We claim that they are also nowhere dense. We will prove this only for $C_h^0$,
since a similar argument works for $C_h^1$. In order to get a contradiction,
assume that $U$ and $V$ are non-empty open subsets of $2^\omega$ such that
$h(x)=\pi_0(g_\alpha(x,y))$
whenever $(x,y)\in (U\times V)\cap T_\alpha$. Fix $q,r,r'\in Q$ such that $r\neq
r'$ and $(q,r),(q,r')\in U\times V$.
Then $\pi_0(g_\alpha(q,r))=h(q)=\pi_0(g_\alpha(q,r'))$, contradicting the
assumption that $\pi_0\upharpoonright (g_\alpha[Q^2])$ is injective.

Since $\MAsigma$ obviously implies $\MAc$, which is equivalent to
$\mathsf{cov}(\textrm{meager})=\cccc$
(see \cite[Theorem 7.13]{blass}), and $|\HH|,|X_\alpha|,|Y_\alpha|<\cccc$ by
condition $(\ref{small})$, there exists $(x,y)\in T_\alpha$ such that
$$
(x,y)\notin g_\alpha^{-1}[\pi_0^{-1}[X_\alpha]\cap T_\alpha]\cup\bigcup_{i\in
2}\pi_i^{-1}[Y_\alpha]\cup\bigcup_{h\in\HH,i\in 2}C_h^i.
$$
Notice that $\HH[X_\alpha]\cap\HH[Y_\alpha]=X_\alpha\cap Y_\alpha=\varnothing$
by conditions $(\ref{disjoint})$, $(\ref{preservation})$, and by our choice of $f$.
The set $g_\alpha^{-1}[\pi_0^{-1}[X_\alpha]\cap T_\alpha]$ guarantees that
$\HH[X_\alpha]\cap \HH[\{\pi_0(g_\alpha(x,y))\}]=\varnothing$.
The sets $\pi_i^{-1}[Y_\alpha]$ guarantee that $\HH[\{x,y\}]\cap
\HH[Y_\alpha]=\varnothing$.
The sets $C_h^i$ guarantee that $\HH[\{x,y\}]\cap
\HH[\{\pi_0(g_\alpha(x,y))\}]=\varnothing$.
Combining the above observations, one sees that
$$
\HH[X_\alpha\cup\{x,y\}]\cap\HH[Y_\alpha\cup \{\pi_0(g_\alpha(x,y))\}]=\varnothing.
$$

It follows that it is possible to construct $X_{\alpha+1}\supseteq
\HH[X_\alpha\cup\{x,y\}]$ and $Y_{\alpha+1}\supseteq \HH[Y_\alpha\cup
\{\pi_0(g_\alpha(x,y))\}]$ that satisfy the
requirement $(\ref{density})$, while still mantaining $(\ref{small})$,
$(\ref{disjoint})$ and $(\ref{preservation})$. This can be done in
$\lambda$ stages, adding one point to each from every non-empty clopen subset of
$2^\omega$ and closing under the action of $\HH$ at each stage.

In the end, set $X=\bigcup_{\alpha\in\cccc}X_\alpha$.

\section{The verification}

We will start by showing that $X$ is $\CDH$. So fix a pair $(A,B)$ of countable
dense subsets of $X$. Since $\cf(\cccc)>\omega$, there exists $\alpha<\cccc$
such that $A\cup B\subseteq X_\alpha$. Now fix $\beta\geq\alpha+1$ such that
$(A,B)=(A_\beta,B_\beta)$. Notice that $X_\beta\setminus (A_\beta\cup B_\beta)$
is $\maxi\{|\beta|,\omega\}$-dense in $2^\omega$ by condition $(\ref{density})$.
Therefore, by condition $(\ref{cdh})$, there exists $f\in\HH_{\beta+1}$ such
that $f[A_\beta]=B_\beta$. Condition $(\ref{preservation})$ guarantees that
$f[X]=X$, so $f\upharpoonright X$ is the desired homeomorphism.

In order to show that $X^2$ is not $\CDH$, we will employ the following
classical result, which is a well-known tool for ``killing'' homeomorphisms (see
\cite{vanmills} for several
interesting applications). For a proof of Theorem \ref{lavrentiev},
see \cite[Theorem 3.9 and Exercise 3.10]{kechris}.
\begin{theorem}[Lavrentiev]\label{lavrentiev}
Let $Z$ be a Polish space and $S\subseteq Z$. Every homeomorphism
$f:S\longrightarrow S$ extends
to a homeomorphism $g:T\longrightarrow T$, where $T\supseteq S$ is a
$\Gd$ subset of $Z$.
\end{theorem}

Let $D$ be a countable dense subset of $X^2$ such that $\pi_0\upharpoonright D$
is injective. Such a subset is easy to construct using the fact that $X$ has no
isolated points. Assume, in order to get a contradiction, that
$f:X^2\longrightarrow X^2$ is a
homeomorphism such that $f[Q^2]=D$.
By Theorem \ref{lavrentiev}, there exists a homeomorphism $g:T\longrightarrow T$
that extends $f$, where $T\supseteq X^2\supseteq Q^2$ is a $\Gd$ subset of
$2^\omega\times 2^\omega$.
Since we enumerated all such homeomorphisms, we must have $g=g_\alpha$ and
$T=T_\alpha$ for some $\alpha <\cccc$. By conditions
$(\ref{kill})$ and $(\ref{disjoint})$, there
exists $(x,y)\in X^2\cap T_\alpha$ such that $\pi_0(g_\alpha(x,y))\notin X$,
contradicting the fact that $g_\alpha(x,y)=g(x,y)=f(x,y)\in X^2$.

\section{Acknowledgements}

The author is grateful to the anonymous referee for spotting
a serious problem in an early version of this article. He also thanks Lyubomyr
Zdomskyy for useful discussions, and in particular for pointing out the footnote
to Theorem \ref{main}.


\begin{thebibliography}{99}

\bibitem{andersoncurtisvanmill}\textsc{R.D. Anderson, D.W. Curtis, J. van Mill.} A fake
topological Hilbert space. \emph{Trans. Amer. Math. Soc.} \textbf{272:1} (1982),
311--321.

\bibitem{arkhangelskiivanmill}\textsc{A.V. Arkhangel$'$ski{\u\i}, J. van Mill.} Topological
homogeneity. To appear in \emph{Recent Progress in General Topology III.}

\bibitem{baldwinbeaudoin}\textsc{S. Baldwin, R.E. Beaudoin.} Countable dense homogeneous
spaces under Martin's axiom. \emph{Israel J. Math.} \textbf{65:2} (1989),
153--164.

\bibitem{blass}\textsc{A. Blass.} Combinatorial cardinal
characteristics of the continuum. \emph{Handbook of Set Theory, vol. 1.}
Springer, Dordrecht, 2010, 395--489.

\bibitem{hernandezgutierrezhrusakvanmill}\textsc{R. Hern\'andez-Guti\'errez, M. Hru\v{s}\'ak, J. van
Mill.} Countable dense homogeneity and $\lambda$-sets. To appear in \emph{Fund.
Math.}

\bibitem{hrusakvanmill}\textsc{M. Hru\v{s}\'ak, J. van Mill.} Nearly countable
dense homogeneous spaces. To appear in \emph{Canad. J. Math.}

\bibitem{hrusakzamoraaviles}\textsc{M. Hru\v{s}\'ak, B. Zamora Avil\'es.} Countable dense
homogeneity of definable spaces. \emph{Proc. Amer. Math. Soc.} \textbf{133:11}
(2005), 3429--3435.

\bibitem{kechris}\textsc{A.S. Kechris.} \emph{Classical descriptive set theory.}
Graduate Texts in Mathematics, 156. Springer-Verlag, New York, 1995. xviii+402
pp.

\bibitem{kunen}\textsc{K. Kunen.} \emph{Set theory.} Studies in Logic (London),
34. College Publications, London, 2011. viii+401 pp.

\bibitem{kunenmedinizdomskyy}\textsc{K. Kunen, A. Medini, L. Zdomskyy.} Seven characterizations
of non-meager P-filters. Preprint.

\bibitem{lawrence}\textsc{L.B. Lawrence.} A rigid subspace of the real line
whose square is a homogeneous subspace of the plane. \emph{Trans. Amer. Math.
Soc.} \textbf{357:7} (2005), 2535--2556.

\bibitem{medinimilovich}\textsc{A. Medini, D. Milovich.} The topology of
ultrafilters as subspaces of $2^\omega$. \emph{Topology Appl.} \textbf{159:5}
(2012), 1318--1333.

\bibitem{vanmillr}\textsc{J. van Mill.} A rigid space $X$ for which $X\times X$
is homogeneous; an application of infinite-dimensional topology. \emph{Proc.
Amer. Math. Soc.} \textbf{83:3} (1981), 597--600.

\bibitem{vanmills}\textsc{J. van Mill.} Sierpi\'{n}ski's technique and subsets
of $\RRR$. \emph{Topology Appl.} \textbf{44:1-3} (1992), 241--261.

\bibitem{vanmilli}\textsc{J. van Mill.} \emph{The infinite-dimensional topology
of function spaces.} North-Holland Mathematical Library, 64. North-Holland
Publishing Co., Amsterdam, 2001. xii+630 pp.

\end{thebibliography}
\end{document}